\documentclass[11pt, twoside, notitlepage]{report}

\usepackage[paper=a4paper,top=2cm,left=2cm,right=2cm,foot=2cm,bottom=3cm]{geometry}

\usepackage[english]{babel}
\addto{\captionsenglish}{%
  
}
\usepackage{amssymb}
\usepackage{amsmath}
\usepackage{alltt}
\usepackage{enumerate}
\usepackage{tabularx}
\usepackage{slashbox}
\usepackage{theorem}
\usepackage{algorithmic, algorithm}
\usepackage{cases}

\usepackage{latexsym}

\newenvironment{proof}{\par\noindent\textit{Proof.}}{$\Box$\par\bigskip\par}
\newenvironment{proofTrees}{\par\noindent\textit{Proof of Theorem~\ref{thm:TreesMain}.}}{$\Box$\par\bigskip\par}
\newenvironment{proofHaircomb}{\par\noindent\textit{Proof of Theorem~\ref{thm:HairComb}.}}{$\Box$\par\bigskip\par}

\newtheorem{theorem}{Theorem}
\newtheorem{observation}{Observation}

\newtheorem{definition}{Definiton}
\newtheorem{lemma}{Lemma}

\usepackage[usenames,dvipsnames]{pstricks}
\usepackage{epsfig}
\usepackage{pst-grad} 
\usepackage{pst-plot} 

\title{\textsc{Homometric sets in trees}\thanks{The first author gratefully acknowledges support from the Swiss National Science Foundation, Grant No. 200021-125287/1.}}
\author{Radoslav Fulek\thanks{
     	Ecole Polytechnique F\'ed\'erale de Lausanne, Switzerland.
		{\tt radoslav.fulek@epfl.ch}}
        \and
		Slobodan Mitrovi\'c\thanks{
     	Ecole Polytechnique F\'ed\'erale de Lausanne, Switzerland.
		{\tt boba5555@gmail.com}     	
     	}
}

\date{}
\begin{document}

\maketitle

\begin{abstract}
 Let $G = (V,E)$ denote a simple graph with the vertex set $V$ and the edge set $E$. The profile of a
vertex set $V'\subseteq V$ denotes the multiset of pairwise distances between the vertices of $V'$. Two disjoint 
subsets of $V$ are \emph{homometric}, if their profiles are the same.  If $G$ is a tree on $n$ vertices we prove that
its vertex sets contains a pair of disjoint homometric subsets of size at least $\sqrt{n/2} - 1$. Previously it was known 
that such a pair of size at least roughly $n^{1/3}$ exists.
We get a better result in case of haircomb trees, in which we are able to find a pair of disjoint homometric sets
of size at least $cn^{2/3}$ for a constant $c > 0$.
\end{abstract}

\section{Introduction}\label{sec:introduction}
    A \emph{graph} is a system of two elements sets called \emph{edges} over a finite set of \emph{vertices}.
    We refer the reader to the book of Diestel~\cite{Diestel} for  definitions of  standard graph notions (such as
    the degree of a vertex, the distance between two vertices etc.) used in the sequel.
	The multiset of pairwise distances of elements of the vertex set in a graph is called the \emph{profile}.
	Two disjoint vertex subsets of a graph are called \emph{homometric sets} if theirs profiles are the same.
	Certain properties of homometric structures that appear in radio communications,
	X-ray crystallography, and self-orthogonal codes are known for at least 80 years~\cite{Wallace, BloomGolomb, RobinsonBernstein}.
	That motivated mathematicians to study homometric sets over various structures, most notably over the set of integers. Here, the graph is an infinite path.

	The following result was claimed by Piccard~\cite{Piccard} in 1939: If two sets of integers, $A$ and $B$, have the same
	multisets of distinct distances, in which each distance occurs at most once, then they are the same up to congruence.
	However, in 1977 Bloom~\cite{Bloom} found an error in the proof and also constructed a counter-example to
	the claim, which is the following: $A = \{0, 1, 4, 10, 12, 17\}$, $B = \{0, 1, 8, 11, 13, 17\}$.
  In 2007 Bekir and Golomb~\cite{BekirG07} showed that no additional counterexamples are possible.

	Homometric sets appear also in  music. A chromatic scale can be seen as the cycle $C_{12}$.
	\emph{The Hexachordal} theorem states that if the nodes of the cycle are divided into two disjoint
	sets $A$ and $B$, each containing exactly six nodes, then the profile of $A$ is equal to the profile of $B$.

   By the \emph{size of homometric sets} we understand the size of a set in the homometric pair.
	In 2010, Albertson, Pach and Young~\cite{AlbertsonPachYoung} initiated the study of homometric sets in graphs.
	They proved that every graph on $n$ vertices contains homometric sets of  size
	at least $\frac{c \log{n}}{\log{\log{n}}}$. On the other hand, they only constructed a class of graphs where the size of homometric sets
    cannot exceed $n/4$, thus, leaving wide open the question of the right order magnitude of the combinatorial bound on the maximal size of homometric sets in the graphs.

     In 2011, Axenovich and \"Ozkahya~\cite{MariaLale} gave a better lower bound on the maximal size of
	 homometric sets in  trees.
	They showed that every tree on $n$ vertices contains  homometric sets
	of size at least $n^{1/3}$. In the same paper they showed that a haircomb tree on $n$ vertices contains
	homometric sets of size at least $\sqrt{n}/2$.
	\\\\
	A \emph{haircomb} $H$ is a tree consisting of a collection of vertex disjoint paths $\{P_1, P_2, \ldots, P_m\}$ such that
	for each $1\leq i< m$ the first vertex of $P_i$ and the first vertex of $P_{i + 1}$ are connected by an edge  as
	illustrated in Figure \ref{fig:haircomb}. The paths $P_1, \ldots, P_m$ are called the legs of $H$. If $P_{i, 1}$
	denotes the first vertex of $P_i$, then the path $S = P_{1, 1} P_{2, 1}, \ldots ,P_{m - 1,1} P_{m, 1}$ is called
	the \emph{spine} of $H$. 
	
			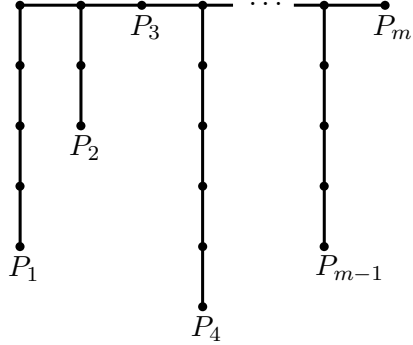
\begin{figure}[h]
			\begin{center}
\scalebox{1} 
{
\begin{pspicture}(0,-2.368125)(5.8428125,2.368125)
\psdots[dotsize=0.12](0.3809375,2.1496875)
\psdots[dotsize=0.12](1.1809375,2.1496875)
\psdots[dotsize=0.12](1.9809375,2.1496875)
\psdots[dotsize=0.12](2.7809374,2.1496875)
\psdots[dotsize=0.12](4.3809376,2.1496875)
\psdots[dotsize=0.12](5.1809373,2.1496875)
\psdots[dotsize=0.12](0.3809375,1.3496875)
\psdots[dotsize=0.12](0.3809375,0.5496875)
\psdots[dotsize=0.12](0.3809375,-0.2503125)
\psdots[dotsize=0.12](1.1809375,1.3496875)
\psdots[dotsize=0.12](1.1809375,0.5496875)
\psdots[dotsize=0.12](0.3809375,-1.0503125)
\psdots[dotsize=0.12](2.7809374,1.3496875)
\psdots[dotsize=0.12](2.7809374,0.5496875)
\psdots[dotsize=0.12](2.7809374,-0.2503125)
\psdots[dotsize=0.12](2.7809374,-1.0503125)
\psdots[dotsize=0.12](2.7809374,-1.8503125)
\psdots[dotsize=0.12](4.3809376,1.3496875)
\psdots[dotsize=0.12](4.3809376,0.5496875)
\psdots[dotsize=0.12](4.3809376,-0.2503125)
\psdots[dotsize=0.12](4.3809376,-1.0503125)
\psline[linewidth=0.04cm](0.3809375,2.1496875)(1.1809375,2.1496875)
\psline[linewidth=0.04cm](1.1809375,2.1496875)(1.9809375,2.1496875)
\psline[linewidth=0.04cm](1.9809375,2.1496875)(2.7809374,2.1496875)
\psline[linewidth=0.04cm](2.7809374,2.1496875)(3.1809375,2.1496875)
\psline[linewidth=0.04cm](4.3809376,2.1496875)(3.9809375,2.1496875)
\psline[linewidth=0.04cm](4.3809376,2.1496875)(5.1809373,2.1496875)
\psline[linewidth=0.04cm](4.3809376,2.1496875)(4.3809376,1.3496875)
\psline[linewidth=0.04cm](4.3809376,0.5496875)(4.3809376,1.3496875)
\psline[linewidth=0.04cm](4.3809376,-0.2503125)(4.3809376,0.5496875)
\psline[linewidth=0.04cm](4.3809376,-1.0503125)(4.3809376,-0.2503125)
\psline[linewidth=0.04cm](2.7809374,1.3496875)(2.7809374,2.1496875)
\psline[linewidth=0.04cm](1.1809375,1.3496875)(1.1809375,2.1496875)
\psline[linewidth=0.04cm](0.3809375,1.3496875)(0.3809375,2.1496875)
\psline[linewidth=0.04cm](0.3809375,0.5496875)(0.3809375,1.3496875)
\psline[linewidth=0.04cm](1.1809375,0.5496875)(1.1809375,1.3496875)
\psline[linewidth=0.04cm](2.7809374,0.5496875)(2.7809374,1.3496875)
\psline[linewidth=0.04cm](2.7809374,-0.2503125)(2.7809374,0.5496875)
\psline[linewidth=0.04cm](2.7809374,-1.0503125)(2.7809374,-0.2503125)
\psline[linewidth=0.04cm](2.7809374,-1.8503125)(2.7809374,-1.0503125)
\psline[linewidth=0.04cm](0.3809375,-0.2503125)(0.3809375,0.5496875)
\psline[linewidth=0.04cm](0.3809375,-1.0503125)(0.3809375,-0.2503125)
\usefont{T1}{ptm}{m}{n}
\rput(5.282344,1.8596874){$P_m$}
\usefont{T1}{ptm}{m}{n}
\rput(4.7123437,-1.3403125){$P_{m - 1}$}
\usefont{T1}{ptm}{m}{n}
\rput(2.8323438,-2.1403124){$P_4$}
\usefont{T1}{ptm}{m}{n}
\rput(2.0323439,1.8596874){$P_3$}
\usefont{T1}{ptm}{m}{n}
\rput(1.2323438,0.2596875){$P_2$}
\usefont{T1}{ptm}{m}{n}
\rput(0.43234375,-1.3403125){$P_1$}
\usefont{T1}{ptm}{m}{n}
\rput(3.6123438,2.1796875){$\ldots$}
\end{pspicture}
}
			\end{center}
			\caption{A haircomb tree.}
			\label{fig:haircomb}
			\end{figure}

	In the present note we further explore homometric sets in trees. 
     In Section~\ref{sec:trees},
	we prove the following theorem thereby improving the previously known best lower bound of $\Omega(n^{1/3})$ on size of homometric sets in trees given in~\cite{MariaLale}.

\begin{theorem}
\label{thm:TreesMain}
Any tree on $n$ vertices contains (a pair of disjoint) homometric sets of size at least $\sqrt{\frac{n}{2}} - \frac{1}{2}$.
\end{theorem}

We were able to obtain a better lower bound in case of haircomb trees whose proof is deferred to Section~\ref{sec:haircomb}.
\begin{theorem}
\label{thm:HairComb}
Any haircomb tree on $n$ vertices contains (a pair of disjoint) homometric sets of size at least $cn^{2/3}$, for a fixed constant $c>0$.
\end{theorem}

\section{Trees}

\label{sec:trees}
	In this section we show that a tree on $n$ vertices contains
	homometric sets of size at least \\ $\sqrt{n/2} - 1$. Moreover, we show a slightly better
    bound for binary trees and prove that our construction, in general, cannot yield a better bound.

    Let us start with some additional definitions.
    By $T=T_r = (V, E)$, where $r \in V$, we denote a tree $T$ rooted at $r$.
	By $h(T_r)$ we denote the height of $T_r$ increased by 1, i.e. the number of vertices on the longest path in $T_r$
	starting at $r$. For example, if $T_r = (V, E)$ and $|V| = 1$, then $h(T_r) = 1$.
	By $T_{\emptyset}$ we denote the empty tree.

	Let $T_r = (V, E)$ and let $v_1, \ldots, v_k$ denote the children of $r$.
	Let $C(T_r) = \{T_{v_1}, \ldots, T_{v_k}\}$ such that
	$r \notin \bigcup_{i = 1}^{k}{V(T_{v_i})}$;
	$\{r\} \cup \bigcup_{i = 1}^{k}{V(T_{v_i})} = V$; and
	$T_{v_i} \neq T_{\emptyset}$ for $1 \le i < k$.
	Let us assume w.l.o.g. that $h(T_{v_i}) \ge h(T_{v_{i + 1}})$ for $1 \le i < k$
	and that $k$ is an even number. Thus, by slightly abusing our notation we allow that $T_{v_k} = T_{\emptyset}$.

    In what follows we construct homometric sets of $T$ of the desired size using a pairing strategy which extends
    the technique from the proof of Theorem 4 in~\cite{MariaLale}.
    For a rooted tree, the pairing strategy in~\cite{MariaLale} gives a one-to-one correspondence between vertices in two constructed homometric sets such that two vertices in each pair are siblings and, thus, have the same distance from the root.
    We extended their approach by pairing paths of the same length that start at siblings.
    More formally, let $V_1,V_2\subseteq V$, such that $V_1\cap V_2=\emptyset$, $|V_1|=|V_2|$, and $V_1$ and $V_2$, resp., induce in $T$
    a collection of vertex disjoint paths $\mathcal{P}_1$ and $\mathcal{P}_2$, resp., so that there exists a bijection between
    $\mathcal{P}_1$ and $\mathcal{P}_2$, which maps a path starting at a vertex $v$ to a path having the same length starting at a sibling of $v$.

    We say that two paths $P_1$ and $P_2$ of $T$ are {\em independent}, if there exists no root-leaf path in $T$ sharing a vertex with both $P_1$ and $P_2$.
    Our construction is based on the following simple observation.

    \begin{observation}\label{obs:pairing-sets}
        If every pair of paths in $\mathcal{P}_1 \cup \mathcal{P}_2$ is independent then $V_1$ and $V_2$ are homometric sets of $T$.
	\end{observation}

    \begin{proof}
    Fix two pairs of vertices $(v_1,v_2)$ and $(u_1,u_2)$ such that vertices in both pairs belong to two
    paired paths, respectively, and have the same distance from the root.
    Since every pair of paths in $\mathcal{P}_1 \cup \mathcal{P}_2$ is independent, the least common ancestor of $v_1$ and $u_1$ has the same distance from the root as the least common ancestor of $v_2$ and $u_2$.
    It follows that the distance between $v_1$ and $u_1$ is the same as the distance between $v_2$ and $u_2$.
    \end{proof}

    In the light of the previous observation the following recursively defined function $f$ gives a lower bound on the size of homometric
    sets obtained by our pairing strategy

	\begin{definition}\label{def:pairing}

		\begin{subnumcases}{f(T_r) \stackrel{\rm{def}}{=} }
			0 & if $|V(T_r)| \le 1$ \label{pairing-1st-case} \\
			\max{\left\{\sum_{i = 1}^{k}{f(T_{v_i})}, \sum_{i = 1}^{\frac{k}{2}}{h(T_{v_{2 i}})}\right\} } &
				if $|V(T_r)| > 1$ and $C(T_r) = \{T_{v_1}, \ldots, T_{v_k}\}$ \label{pairing-2nd-case}
		\end{subnumcases}
	\end{definition}
	Let $P_{v_j}$ be a longest path in $T_{v_j}$ such that $v_j$ is one of its ends.
	The summand $h(T_{v_{2 i}})$ in Definition~\ref{def:pairing} accounts for
	the pairing of the path $P_{v_{2 i}}$ with the subpath of
	$P_{v_{2 i - 1}}$ starting at $v_{2 i - 1}$ of the size $|P_{v_{2 i}}|$.
	Note that $|P_{v_{2 i - 1}}| \ge |P_{v_{2 i}}|$ by the assumption.
	An illustration of the pairing strategy specified in Definition~\ref{def:pairing} is provided in
	Figure \ref{fig:pairing-strategy}.
	
	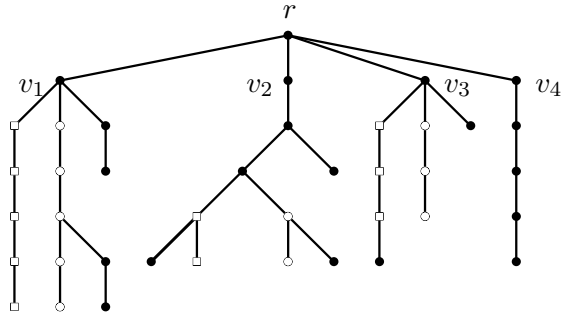
\begin{figure}[h]
	\begin{center}
\scalebox{1} 
{
\begin{pspicture}(0,-2.0792189)(7.7028127,2.0992188)
\psline[linewidth=0.04cm](2.5809374,-0.7992188)(1.9809375,-1.3992188)
\psline[linewidth=0.03cm](3.7809374,-0.7992188)(4.3809376,-1.3992188)
\psline[linewidth=0.03cm](4.9809375,-0.7992188)(4.9809375,-1.3992188)
\psline[linewidth=0.03cm](4.9809375,-0.19921875)(4.9809375,-0.7992188)
\psline[linewidth=0.03cm](4.9809375,0.40078124)(4.9809375,-0.19921875)
\psline[linewidth=0.03cm](5.5809374,1.0007813)(4.9809375,0.40078124)
\psline[linewidth=0.03cm](0.7809375,1.0007813)(0.1809375,0.40078124)
\psline[linewidth=0.03cm](0.1809375,-1.3992188)(0.1809375,-1.9992187)
\psline[linewidth=0.03cm](0.1809375,-0.7992188)(0.1809375,-1.3992188)
\psline[linewidth=0.03cm](0.1809375,-0.19921875)(0.1809375,-0.7992188)
\psline[linewidth=0.03cm](0.1809375,0.40078124)(0.1809375,-0.19921875)
\psline[linewidth=0.03cm](2.5809374,-0.7992188)(2.5809374,-1.3992188)
\psline[linewidth=0.03cm](3.1809375,-0.19921875)(2.5809374,-0.7992188)
\psline[linewidth=0.03cm](3.1809375,-0.19921875)(3.7809374,-0.7992188)
\psline[linewidth=0.03cm](5.5809374,-0.19921875)(5.5809374,-0.7992188)
\psline[linewidth=0.03cm](5.5809374,-0.19921875)(5.5809374,0.40078124)
\psline[linewidth=0.03cm](5.5809374,1.0007813)(5.5809374,0.40078124)
\psline[linewidth=0.03cm](3.7809374,-0.7992188)(3.7809374,-1.3992188)
\psline[linewidth=0.03cm](0.7809375,1.0007813)(0.7809375,0.40078124)
\psline[linewidth=0.03cm](1.3809375,-1.9992187)(1.3809375,-1.3992188)
\psline[linewidth=0.03cm](0.7809375,-0.7992188)(1.3809375,-1.3992188)
\psline[linewidth=0.03cm](0.7809375,0.40078124)(0.7809375,-0.19921875)
\psline[linewidth=0.03cm](0.7809375,-1.3992188)(0.7809375,-1.9992187)
\psline[linewidth=0.03cm](0.7809375,-0.7992188)(0.7809375,-1.3992188)
\psline[linewidth=0.03cm](0.7809375,-0.19921875)(0.7809375,-0.7992188)
\psdots[dotsize=0.12](3.7809374,1.6007812)
\psdots[dotsize=0.12](0.7809375,1.0007813)
\psdots[dotsize=0.12,fillstyle=solid,dotstyle=square](0.1809375,0.40078124)
\psdots[dotsize=0.12,fillstyle=solid,dotstyle=o](0.7809375,0.40078124)
\psdots[dotsize=0.12](1.3809375,0.40078124)
\psdots[dotsize=0.12](1.3809375,-0.19921875)
\psdots[dotsize=0.12,fillstyle=solid,dotstyle=o](0.7809375,-0.19921875)
\psdots[dotsize=0.12,fillstyle=solid,dotstyle=square](0.1809375,-0.19921875)
\psdots[dotsize=0.12,fillstyle=solid,dotstyle=o](0.7809375,-0.7992188)
\psdots[dotsize=0.12,fillstyle=solid,dotstyle=square](0.1809375,-0.7992188)
\psdots[dotsize=0.12,fillstyle=solid,dotstyle=square](0.1809375,-1.3992188)
\psdots[dotsize=0.12,fillstyle=solid,dotstyle=o](0.7809375,-1.3992188)
\psdots[dotsize=0.12](1.3809375,-1.3992188)
\psdots[dotsize=0.12](1.3809375,-1.9992187)
\psdots[dotsize=0.12,fillstyle=solid,dotstyle=o](0.7809375,-1.9992187)
\psdots[dotsize=0.12,fillstyle=solid,dotstyle=square](0.1809375,-1.9992187)
\psline[linewidth=0.03cm](0.7809375,1.0007813)(1.3809375,0.40078124)
\psline[linewidth=0.03cm](1.3809375,0.40078124)(1.3809375,-0.19921875)
\psdots[dotsize=0.12](3.7809374,1.0007813)
\psdots[dotsize=0.12](3.7809374,0.40078124)
\psdots[dotsize=0.12](4.3809376,-0.19921875)
\psdots[dotsize=0.12](3.1809375,-0.19921875)
\psdots[dotsize=0.12,fillstyle=solid,dotstyle=o](3.7809374,-0.7992188)
\psdots[dotsize=0.12,fillstyle=solid,dotstyle=o](3.7809374,-1.3992188)
\psdots[dotsize=0.12](4.3809376,-1.3992188)
\psdots[dotsize=0.12,fillstyle=solid,dotstyle=square](2.5809374,-1.3992188)
\psdots[dotsize=0.12](5.5809374,1.0007813)
\psdots[dotsize=0.12,fillstyle=solid,dotstyle=square](4.9809375,0.40078124)
\psdots[dotsize=0.12,fillstyle=solid,dotstyle=o](5.5809374,0.40078124)
\psdots[dotsize=0.12](6.1809373,0.40078124)
\psdots[dotsize=0.12,fillstyle=solid,dotstyle=square](4.9809375,-0.19921875)
\psdots[dotsize=0.12](6.7809377,0.40078124)
\psdots[dotsize=0.12](6.7809377,-0.19921875)
\psdots[dotsize=0.12](6.7809377,1.0007813)
\psline[linewidth=0.03cm](6.7809377,1.0007813)(6.7809377,0.40078124)
\psline[linewidth=0.03cm](6.7809377,0.40078124)(6.7809377,-0.19921875)
\psdots[dotsize=0.12,fillstyle=solid,dotstyle=o](5.5809374,-0.19921875)
\psdots[dotsize=0.12,fillstyle=solid,dotstyle=o](5.5809374,-0.7992188)
\psdots[dotsize=0.12,fillstyle=solid,dotstyle=square](4.9809375,-0.7992188)
\psdots[dotsize=0.12](4.9809375,-1.3992188)
\psdots[dotsize=0.12](6.7809377,-0.7992188)
\psdots[dotsize=0.12](6.7809377,-1.3992188)
\psline[linewidth=0.03cm](6.7809377,-0.19921875)(6.7809377,-0.7992188)
\psline[linewidth=0.03cm](6.7809377,-0.7992188)(6.7809377,-1.3992188)
\psline[linewidth=0.03cm](3.7809374,1.6007812)(3.7809374,1.0007813)
\psline[linewidth=0.03cm](3.7809374,1.0007813)(3.7809374,0.40078124)
\psline[linewidth=0.03cm](3.7809374,0.40078124)(3.1809375,-0.19921875)
\psline[linewidth=0.03cm](3.7809374,0.40078124)(4.3809376,-0.19921875)
\psdots[dotsize=0.12,fillstyle=solid,dotstyle=square](2.5809374,-0.7992188)
\psline[linewidth=0.03cm](0.7809375,1.0007813)(3.7809374,1.6007812)
\psline[linewidth=0.03cm](3.7809374,1.6007812)(5.5809374,1.0007813)
\psline[linewidth=0.03cm](3.7809374,1.6007812)(6.7809377,1.0007813)
\psline[linewidth=0.03cm](5.5809374,1.0007813)(6.1809373,0.40078124)
\psdots[dotsize=0.12](1.9809375,-1.3992188)
\usefont{T1}{ptm}{m}{n}
\rput(3.8023438,1.9107813){$r$}
\usefont{T1}{ptm}{m}{n}
\rput(0.41234374,0.91078126){$v_1$}
\usefont{T1}{ptm}{m}{n}
\rput(3.4123437,0.91078126){$v_2$}
\usefont{T1}{ptm}{m}{n}
\rput(6.012344,0.91078126){$v_3$}
\usefont{T1}{ptm}{m}{n}
\rput(7.2123437,0.91078126){$v_4$}
\end{pspicture}
}

	\end{center}
	\caption{An illustration of the pairing strategy implied by Definition \ref{def:pairing}. The set of the empty circles and
	the set of the empty squares represent two disjoint homometric sets of the size 10. Note that in this example
	$\sum_{i = 1}^4{f(T_{v_i})} = \sum_{i = 1}^2{h(T_{v_{2 i}})}$.
	}
	\label{fig:pairing-strategy}
	\end{figure}
	
Hence, by the previous paragraph and Observation~\ref{obs:pairing-sets} we have the following.

	\begin{lemma}\label{lemma:pairing-sets}
		A tree $T_r$ contains two disjoint homometric sets of size at least $f(T_r)$.
	\end{lemma}
	
The main ingredient of the proof of Theorem~\ref{thm:TreesMain} is the next lemma.

	\begin{lemma}\label{lemma:pairing-value}
		For every tree $T_r = (V, E)$
			\begin{equation}\label{pairing-inequality}
				f(T_r) \ge \frac{|V|}{2 h(T_r)} - \frac{1}{2}
			\end{equation}
	\end{lemma}
	\begin{proof}
		We give a proof by induction on the height of $T_r$.
		\begin{description}
			\item[Basic step.] Let $T_r$ be a tree, and $h(T_r) = 1$. By the definition of the function $f$ it follows:
				\[
					f(T_r) \stackrel{\eqref{pairing-1st-case}}{=} 0
				\]
				On the other hand
				\[
					\frac{|V(T_r)|}{2 h(T_r)} - \frac{1}{2} = 0
				\]

			\item[Inductive step.] Let $T_r$ be a tree on $n$ vertices, and $g = h(T_r) \ge 2$.
				We assume that for every tree $T_{r'}$ such
				that $h(T_{r'}) < g$ the inequality~\eqref{pairing-inequality} holds, and we
				prove that~\eqref{pairing-inequality} holds for $T_r$ as well.
				\\
				
				Let $C(T_r) = \{T_{v_1}, \ldots, T_{v_k}\}$, and let $g_i$ denote $h(T_{v_i})$, for $i = 1 \ldots k$.
				Let $S = \sum_{i = 2}^{k}{g_i}$. We consider two cases:
				\begin{description}
					\item[Case 1.] Let us assume $S \ge \frac{n}{g} - 1$. Let
						$\Delta = \sum_{i = 1}^{\frac{k}{2}}{(g_{2 i - 1} - g_{2 i})}$. We bound $\Delta$
						as follows:
						\begin{equation}\label{delta-bound}
							\Delta = \sum_{i = 1}^{\frac{k}{2}}{(g_{2 i - 1} - g_{2 i})}
										 \le g_1 - g_2 + \sum_{i = 2}^{\frac{k}{2}}{(g_{2 i - 2} - g_{2 i})}
										 = g_1 - g_k
										 \le g_1
						\end{equation}
						By~\eqref{pairing-2nd-case} in Definition~\ref{def:pairing} we have:
						\begin{equation}\label{case1-pairing-bound}
							f(T_r) \ge \sum_{i = 1}^{\frac{k}{2}}{g_{2 i}}
							= \sum_{i = 1}^{\frac{k}{2}}{\frac{g_{2i - 1} + g_{2 i} - (g_{2i - 1} - g_{2 i})}{2}}
							= \frac{g_1 + S - \Delta}{2}
						\end{equation}
						The upper-bound~\eqref{delta-bound} along with~\eqref{case1-pairing-bound} implies:
						\[
							f(T_r) \ge \frac{S}{2} \ge \frac{n}{2 g} - \frac{1}{2}
						\]
						Therefore, \eqref{pairing-inequality} holds in this case.
						
					\item[Case 2.] Let us assume $S < \frac{n}{g} - 1$. Let $n_i = V(T_{v_i})$.
					 	By~\eqref{pairing-2nd-case} in Definition~\ref{def:pairing} and the inductive hypothesis we conclude:
						\begin{equation}\label{pairing-case2}
							f(T_r) \ge \sum_{i = 1}^{k}{f(T_{v_i})}
							\ge \sum_{i = 1}^{k}{\frac{n_i - g_i}{2 g_i}}
							\ge \sum_{i = 1}^{k}{\frac{n_i - g_i}{2 g_1}}
							= \frac{(n - 1) - (g_1 + S)}{2 g_1}
						\end{equation}
						Applying the condition of the present case and observation $g_1 = g - 1$ to
						the expression~\eqref{pairing-case2}, we obtain the following:
						\[
							f(T_r) \ge \frac{n - 1 - \big(\frac{n}{g} - 1\big)}{2 g_1} - \frac{1}{2}
							= \frac{\frac{n(g - 1)}{g}}{2 g_1} - \frac{1}{2}
							= \frac{n}{2 g} - \frac{1}{2}
						\]
				\end{description}
		\end{description}
		This completes the proof.
	\end{proof}
	

Finally, we are in a position to give the proof of the main result of this section. \\

	\begin{proofTrees}
		Let $T$ be a tree on $n$ vertices. Let $P = v_1 v_2 \ldots v_g$ be a longest path in $T$. If $g$ is an even number
		let $g' = g$, otherwise $g' = g - 1$.
		If $g \ge \sqrt{2 n}$ then  $\{v_1, \ldots, v_{g'/2}\}$ and
		 $\{v_{g'/2 + 1}, \ldots, v_{g'}\}$ are two disjoint homometric sets of
		size at least $\sqrt{n/2}$.
		\\\\
		If $g < \sqrt{2 n}$ let $r = v_{g'/2 + 1}$ and consider $T_r$. Since
		$h(T_r) \le \sqrt{n/2}$, by Lemma~\ref{lemma:pairing-sets} and
		Lemma~\ref{lemma:pairing-value}
		we conclude that $T_r$ contains two disjoint homometric sets of size at least
		\[
			f(T_r)
			\ge \frac{n}{2 h(T_r)} - \frac{1}{2}
			\ge \frac{n}{2 \sqrt{\frac{n}{2}}} - \frac{1}{2}
			= \sqrt{\frac{n}{2}} - \frac{1}{2}
		\]
		
		This concludes the proof.
	\end{proofTrees}
	
\subsection{Binary trees}\label{sec:binary-trees}
	Observe that for binary trees the value $\Delta$ defined in Lemma~\ref{lemma:pairing-value} is
	equal to $g_1 - g_2$. Following the proof of Lemma~\ref{lemma:pairing-value} this observation gives
	\begin{equation}\label{binary-tree-pairing}
		f(T_r) \ge \frac{|V(T_r)|}{h(T_r)} - 1
	\end{equation}
	Additionally, the inequality \eqref{binary-tree-pairing} implies that every binary tree on $n$ vertices
	contains two disjoint homometric sets of size at least $\sqrt{n} - 1$.
	\\

	Furthermore, in Figure \ref{fig:tight-example-binary-trees} we define a family of binary trees
	$\{R_i\}_{i \in \mathbb{N}}$ that shows the lower bound \eqref{binary-tree-pairing} is tight.

		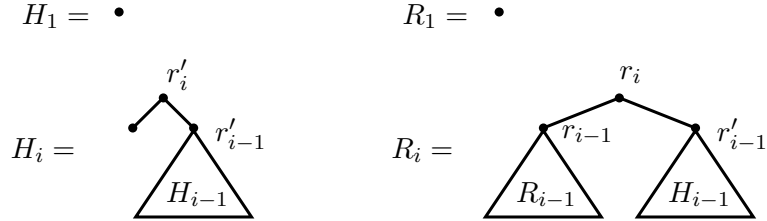
\begin{figure}[h]
		\begin{center}
\scalebox{1} 
{
\begin{pspicture}(0,-1.4192188)(10.682813,1.4592187)
\usefont{T1}{ptm}{m}{n}
\rput(0.77234375,1.2707813){$H_1 = $}
\psdots[dotsize=0.12](1.6009375,1.3207812)
\usefont{T1}{ptm}{m}{n}
\rput(5.592344,-0.50921875){$R_i = $}
\usefont{T1}{ptm}{m}{n}
\rput(8.332344,0.49078125){$r_i$}
\usefont{T1}{ptm}{m}{n}
\rput(5.762344,1.2707813){$R_1 = $}
\psdots[dotsize=0.12](6.6009374,1.3207812)
\usefont{T1}{ptm}{m}{n}
\rput(0.60234374,-0.50921875){$H_i = $}
\usefont{T1}{ptm}{m}{n}
\rput(2.3723438,0.49078125){$r'_i$}
\usefont{T1}{ptm}{m}{n}
\rput(2.6323438,-1.1092187){$H_{i - 1}$}
\usefont{T1}{ptm}{m}{n}
\rput(7.222344,-1.1092187){$R_{i - 1}$}
\usefont{T1}{ptm}{m}{n}
\rput(9.232344,-1.1092187){$H_{i - 1}$}
\usefont{T1}{ptm}{m}{n}
\rput(7.762344,-0.30921876){$r_{i - 1}$}
\usefont{T1}{ptm}{m}{n}
\rput(9.802343,-0.30921876){$r'_{i - 1}$}
\usefont{T1}{ptm}{m}{n}
\rput(3.2023437,-0.30921876){$r'_{i - 1}$}
\psdots[dotsize=0.12](1.7809376,-0.21921875)
\psdots[dotsize=0.12](2.1809375,0.18078125)
\psdots[dotsize=0.12](2.5809374,-0.21921875)
\psline[linewidth=0.04cm](2.1809375,0.18078125)(2.5809374,-0.21921875)
\psline[linewidth=0.04cm](2.1809375,0.18078125)(1.7809376,-0.21921875)
\pstriangle[linewidth=0.04,dimen=outer](2.5809374,-1.4192188)(1.6,1.2)
\psdots[dotsize=0.12](8.180938,0.18078125)
\psdots[dotsize=0.12](7.1809373,-0.21921875)
\psdots[dotsize=0.12](9.180938,-0.21921875)
\pstriangle[linewidth=0.04,dimen=outer](7.1809373,-1.4192188)(1.6,1.2)
\pstriangle[linewidth=0.04,dimen=outer](9.180938,-1.4192188)(1.6,1.2)
\psline[linewidth=0.04cm](8.180938,0.18078125)(7.1809373,-0.21921875)
\psline[linewidth=0.04cm](8.180938,0.18078125)(9.180938,-0.21921875)
\end{pspicture}
}
		\end{center}
		\caption{Definition of family of binary trees that shows the lower bound \eqref{binary-tree-pairing} is tight.}
		\label{fig:tight-example-binary-trees}
		\end{figure}
		
	\begin{lemma}\label{lemma:Hi-pairing}
		For a binary tree $H_i$, defined in Figure \ref{fig:tight-example-binary-trees},
		\begin{equation*}
			f(H_i) =
				\begin{cases}
				 0 & \rm{if\ } $i = 1$ \\
				 1 & \rm{otherwise}
				\end{cases}
		\end{equation*}
	\end{lemma}
	\begin{proof}
		By induction on $i$.
	\end{proof}

	\begin{lemma}\label{lemma:Ri-pairing}
		For a binary tree $R_i$, defined in Figure \ref{fig:tight-example-binary-trees},
		\[
			f(R_i) = i - 1
		\]
	\end{lemma}
	\begin{proof}
		We give a proof by induction on $i$.
		\begin{description}
			\item[Basic step.] Let $i = 1$. Then
				\[
					f(R_1) \stackrel{\eqref{pairing-1st-case}}{=} 0 = i - 1
				\]
				For $i = 2$ in the similar way we obtain:
				\begin{eqnarray*}
					f(R_2) & \stackrel{\eqref{pairing-2nd-case}}{=} & \max{\{f(R_1) + f(H_1), h(R_1)\}} = 1
				\end{eqnarray*}
				
			\item[Inductive step.] Let $i > 2$. We assume $f(R_j) = j - 1$ for every $1 \le j < i$ and prove
				$f(R_i) = i - 1$. Observing $h(H_i) = h(R_i) = i$ and recalling Lemma~\ref{lemma:Hi-pairing}
				we conclude:
				\begin{eqnarray*}
					f(R_i) & \stackrel{\eqref{pairing-2nd-case}}{=} & \max{\{f(R_{i - 1}) + f(H_{i - 1}), h(R_{i - 1})\}} \\
								& = & \max{\{(i - 1 - 1) + 1, (i - 1)\}} \\
								& = & i - 1
				\end{eqnarray*}
		\end{description}
		
		This completes the proof.
	\end{proof}
	
	\begin{theorem}
		For every $n_0$ there exists a binary tree $T_r$ on $n \ge n_0$ vertices such that
		\[
			f(T_r) = \left\lceil \frac{n}{h(T_r)} - 1\right\rceil,
		\]
		i.e. \eqref{binary-tree-pairing} is a tight lower bound on size of disjoint homometric
		sets obtained over binary trees.
	\end{theorem}
	\begin{proof}
		By induction on $i$ one can trivially prove that $|V(R_i)| = i (i - 1) + 1$ and $|V(H_i)| = 2 i - 1$.
		Following that observation and recalling Lemma \ref{lemma:Ri-pairing} we conclude the proof.
	\end{proof}

\section{Haircomb}\label{sec:haircomb}

	In this section we prove that every haircomb tree on $n$ vertices contains two disjoint homometric sets
	of size at least $cn^{2/3}$ where $c>0$ is a constant. \\

\begin{proofHaircomb}
    Let $H$ denote a haircomb on $n$ vertices with the spine length $s>0$.
    Let $l>0$ denote the length of its longest leg. Let us assume that $l<n/2$ as otherwise we are done.

        Let $L_1,\ldots, L_s$ denote the legs of $H$ ordered from the longest to the shortest one.

        Let $H_o$ and $H_e$ denote the subgraphs of $H$ such that $H_o$ (resp. $H_e$) consists of the spine and all the legs $L_{i}$
        with an odd (resp. even) index $i$.
        We get homemetric sets of the required size in $H$ by studying the overlaps of certain drawings of $H_o \uplus H_e$ ($\uplus$ stands for the disjoint union), in which their vertices are represented by the points of the integer lattice $\mathbb{Z}^2$.
        We consider the family  $D_r$, $0\le r \le s$ of drawings of $H_o \uplus H_e$.
        Let $v_1,..,v_s$ denote the vertices of the spine of $H$ so that $v_iv_{i+1}\in E(H)$.
        Let $P_{i,j}$ denote the $j$-th vertex of the leg of $H$ starting at $v_i$.

        In the drawing $D_0$  the leg that starts at $v_i \in H_e$ is drawn so that
        $P_{i,j}$ is mapped to the point $(i,j)$, if it belongs to $H_e$, and to the point $(i+s,j)$, otherwise.
        In particular, the vertex $v_i$ of the spine in $H_e$ is mapped to $(i,1)$, and in $H_o$ to $(v+s,1)$.

        We obtain the drawing $D_r$ of $H_o \uplus H_e$ by shifting the drawing of $H_e$ in  $D_0$
        by $r$ units to the right.

        We define $O_r$ to be the number of lattice points in $\mathbb{Z}^2$ representing both a vertex from a leg $L_i$ with
        an odd index $i$ and a vertex from a leg $L_i$ with an even index $i$ in $D_r$.
		Clearly, for a given drawing $D_r$ the overlapped vertices, which we count by $O_r$,
        give rise to homometric sets of size $O_r$.
		We proceed by counting the total number of overlaps in the drawings of the set
		$\mathcal{D}=\{D_0, D_1, \ldots, D_{2 s}\}$. Then by averaging we get a drawing $D_r$ with a big overlap.
		An overlapping is illustrated in Figure \ref{fig:haircomb-2-3-overlappings}.
			\begin{figure}[h]
			\begin{center}
\scalebox{0.8} 
{
\begin{pspicture}(0,-6.204531)(10.827188,6.184531)
\definecolor{color384}{rgb}{0.8,0.8,0.8}
\psline[linewidth=0.03cm](10.383125,2.1045313)(10.383125,1.1045313)
\psline[linewidth=0.03cm](9.383125,1.1045313)(9.383125,2.1045313)
\psline[linewidth=0.03cm](8.383125,1.1045313)(8.383125,6.1045313)
\psline[linewidth=0.03cm](7.383125,4.1045313)(7.383125,1.1045313)
\psline[linewidth=0.03cm](5.383125,1.1045313)(5.383125,4.1045313)
\psline[linewidth=0.03cm](4.383125,2.1045313)(4.383125,1.1045313)
\psline[linewidth=0.03cm](3.383125,1.1045313)(3.383125,3.1045313)
\psline[linewidth=0.03cm](2.383125,1.1045313)(10.383125,1.1045313)
\psline[linewidth=0.03cm](2.383125,5.1045313)(2.383125,1.1045313)
\psdots[dotsize=0.12,fillstyle=solid,dotstyle=o](2.383125,5.1045313)
\psdots[dotsize=0.12,fillstyle=solid,dotstyle=o](2.383125,4.1045313)
\psdots[dotsize=0.12,fillstyle=solid,dotstyle=o](2.383125,3.1045313)
\psdots[dotsize=0.12,fillstyle=solid,dotstyle=o](2.383125,2.1045313)
\psdots[dotsize=0.12,fillstyle=solid,dotstyle=o](2.383125,1.1045313)
\psdots[dotsize=0.12](3.383125,1.1045313)
\psdots[dotsize=0.12](3.383125,2.1045313)
\psdots[dotsize=0.12](3.383125,3.1045313)
\psdots[dotsize=0.12,fillstyle=solid,dotstyle=o](4.383125,1.1045313)
\psdots[dotsize=0.12,fillstyle=solid,dotstyle=o](4.383125,2.1045313)
\psdots[dotsize=0.12](5.383125,1.1045313)
\psdots[dotsize=0.12](5.383125,2.1045313)
\psdots[dotsize=0.12](5.383125,3.1045313)
\psdots[dotsize=0.12](5.383125,4.1045313)
\psdots[dotsize=0.12](6.383125,1.1045313)
\psdots[dotsize=0.12,fillstyle=solid,dotstyle=o](7.383125,1.1045313)
\psdots[dotsize=0.12,fillstyle=solid,dotstyle=o](7.383125,2.1045313)
\psdots[dotsize=0.12,fillstyle=solid,dotstyle=o](7.383125,3.1045313)
\psdots[dotsize=0.12,fillstyle=solid,dotstyle=o](7.383125,4.1045313)
\psdots[dotsize=0.12](8.383125,1.1045313)
\psdots[dotsize=0.12](8.383125,2.1045313)
\psdots[dotsize=0.12](8.383125,3.1045313)
\psdots[dotsize=0.12](8.383125,4.1045313)
\psdots[dotsize=0.12](8.383125,5.1045313)
\psdots[dotsize=0.12](8.383125,6.1045313)
\psdots[dotsize=0.12](9.383125,1.1045313)
\psdots[dotsize=0.12](9.383125,2.1045313)
\psdots[dotsize=0.12,fillstyle=solid,dotstyle=o](10.383125,1.1045313)
\psdots[dotsize=0.12,fillstyle=solid,dotstyle=o](10.383125,2.1045313)
\usefont{T1}{ptm}{m}{n}
\rput(2.3767188,0.7995312){\footnotesize $L_2$}
\usefont{T1}{ptm}{m}{n}
\rput(3.3767188,0.7995312){\footnotesize $L_5$}
\usefont{T1}{ptm}{m}{n}
\rput(4.3767185,0.7995312){\footnotesize $L_6$}
\usefont{T1}{ptm}{m}{n}
\rput(5.3767185,0.7995312){\footnotesize $L_3$}
\usefont{T1}{ptm}{m}{n}
\rput(6.3767185,0.7995312){\footnotesize $L_9$}
\usefont{T1}{ptm}{m}{n}
\rput(7.3767185,0.7995312){\footnotesize $L_4$}
\usefont{T1}{ptm}{m}{n}
\rput(8.3767185,0.7995312){\footnotesize $L_1$}
\usefont{T1}{ptm}{m}{n}
\rput(9.3767185,0.7995312){\footnotesize $L_7$}
\usefont{T1}{ptm}{m}{n}
\rput(10.3767185,0.7995312){\footnotesize $L_8$}
\psline[linewidth=0.03cm](9.383125,-5.695469)(9.383125,-4.695469)
\psline[linewidth=0.03cm](8.383125,-5.695469)(8.383125,-0.6954687)
\psline[linewidth=0.03cm](5.383125,-5.695469)(5.383125,-2.6954687)
\psline[linewidth=0.03cm](2.383125,-4.695469)(2.383125,-5.695469)
\psline[linewidth=0.03cm](3.383125,-5.695469)(3.383125,-3.6954687)
\psline[linewidth=0.03cm](0.383125,-5.695469)(9.383125,-5.695469)
\psline[linewidth=0.03cm](0.383125,-1.6954688)(0.383125,-5.695469)
\psdots[dotsize=0.12,fillstyle=solid,dotstyle=o](0.383125,-1.6954688)
\psdots[dotsize=0.12,fillstyle=solid,dotstyle=o](0.383125,-2.6954687)
\psdots[dotsize=0.12,fillstyle=solid,dotstyle=o](0.383125,-3.6954687)
\psdots[dotsize=0.12,fillstyle=solid,dotstyle=o](0.383125,-4.695469)
\psdots[dotsize=0.12,fillstyle=solid,dotstyle=o](0.383125,-5.695469)
\psdots[dotsize=0.12](3.383125,-5.695469)
\psdots[dotsize=0.12](3.383125,-4.695469)
\psdots[dotsize=0.12](3.383125,-3.6954687)
\psdots[dotsize=0.12,fillstyle=solid,dotstyle=o](2.383125,-5.695469)
\psdots[dotsize=0.12,fillstyle=solid,dotstyle=o](2.383125,-4.695469)
\psdots[dotsize=0.12,linecolor=color384](5.383125,-5.695469)
\psdots[dotsize=0.12,linecolor=color384](5.383125,-4.695469)
\psdots[dotsize=0.12,linecolor=color384](5.383125,-3.6954687)
\psdots[dotsize=0.12,linecolor=color384](5.383125,-2.6954687)
\psdots[dotsize=0.12](6.383125,-5.695469)
\psdots[dotsize=0.12,linecolor=color384](8.383125,-5.695469)
\psdots[dotsize=0.12,linecolor=color384](8.383125,-4.695469)
\psdots[dotsize=0.12](8.383125,-3.6954687)
\psdots[dotsize=0.12](8.383125,-2.6954687)
\psdots[dotsize=0.12](8.383125,-1.6954688)
\psdots[dotsize=0.12](8.383125,-0.6954687)
\psdots[dotsize=0.12](9.383125,-5.695469)
\psdots[dotsize=0.12](9.383125,-4.695469)
\usefont{T1}{ptm}{m}{n}
\rput(0.37671876,-6.0004687){\footnotesize $L_2$}
\usefont{T1}{ptm}{m}{n}
\rput(3.3767188,-6.0004687){\footnotesize $L_5$}
\usefont{T1}{ptm}{m}{n}
\rput(2.3767188,-6.0004687){\footnotesize $L_6$}
\usefont{T1}{ptm}{m}{n}
\rput(5.306719,-6.0004687){\footnotesize $L_3, L_4$}
\usefont{T1}{ptm}{m}{n}
\rput(6.3767185,-6.0004687){\footnotesize $L_9$}
\usefont{T1}{ptm}{m}{n}
\rput(8.306719,-6.0004687){\footnotesize $L_1, L_8$}
\usefont{T1}{ptm}{m}{n}
\rput(9.3767185,-6.0004687){\footnotesize $L_7$}
\psline[linewidth=0.03cm,arrowsize=0.05291667cm 2.0,arrowlength=1.4,arrowinset=0.4]{->}(5.983125,0.30453125)(5.983125,-0.29546875)
\usefont{T1}{ptm}{m}{n}
\rput(6.2545314,0.01453125){$D_7$}
\end{pspicture}
}

			\end{center}
			\caption{An example of overlapping. The sketch below represents the overlapping for $D_7$, for which $O_7 = 6$.
				The gray vertices are the overlapped ones.}
			\label{fig:haircomb-2-3-overlappings}
			\end{figure}
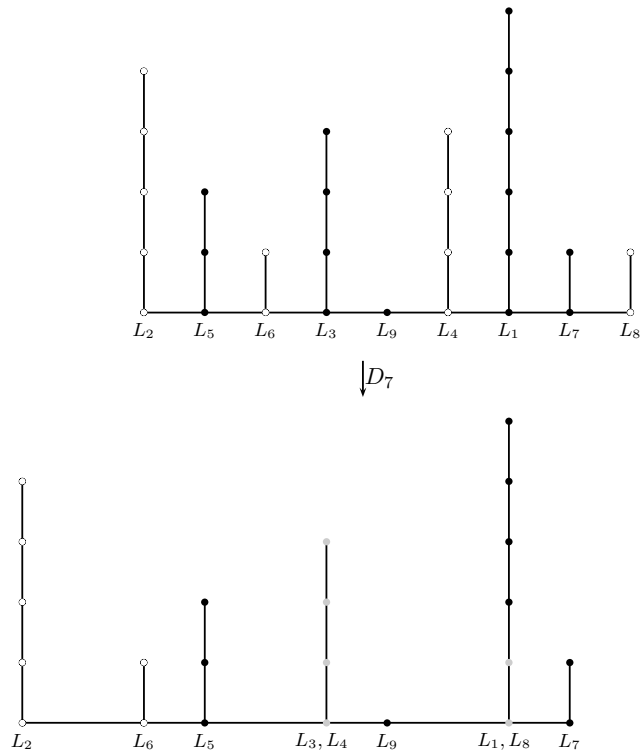

			Observe that a leg $L_i$ with an odd index overlaps with a leg in $L_i$ with an even index in exactly one
			drawing of $\mathcal{D}$.
            Moreover, the size of such overlap is equal to the size of the shorter leg among the two overlapping legs.
            Hence, the total number of overlaps $\mathbb{O}$ in the drawings of $\mathcal{D}$ can be lower bounded by
            \begin{equation}\label{eq:haircomb-o-bound}
                           \mathbb{O}\ge \sum_{i=1}^{\left\lfloor s/2\right\rfloor}i|L_{2i}|
            \end{equation}

            The last inequality can be understood as weighted sum, $i$'s being weights,
            of values $|L_{2 i}|$, for $i = 1, \ldots, \left\lfloor s/2\right\rfloor$.
            Note that $\sum_{i=1}^{\left\lfloor s/2\right\rfloor}(|L_{2i - 1}| - |L_{2i}|) \le l$, and therefore
            the right-hand side of inequality \eqref{eq:haircomb-o-bound} is at least
            $|L_{2}|\sum_{i=1}^{\frac{1}{|L_{2}|}\left\lfloor\frac{n-l}{2}\right\rfloor}i$.
            Thus, we can lower bound $\mathbb{O}$ as follows:
            \begin{eqnarray}
            \nonumber
            \mathbb{O}  & \ge & \sum_{i=1}^{\left\lfloor s/2\right\rfloor}i|L_{2i}|
            	\ge |L_{2}|\sum_{i=1}^{\frac{\lfloor n/4 \rfloor}{|L_{2}|}}i
            	\ge |L_{2}|{\frac 12}\left({\frac{n}{4|L_{2}|}}\right)^2= {\frac{n^2}{32|L_{2}|}} \\
            \nonumber
            & \ge & {\frac{n^2}{32\, l}}
            \end{eqnarray}

            The lower bound ${\frac{n^2}{32\, l}}$ on $\mathbb{O}$ immediately implies that there exists a drawing of
            $H_o \uplus H_e$ in the family of drawings $\mathcal{D}$
            giving an overlap of size ${\frac{n^2}{32\,l\,s}}$. Thus, we can always get homometric sets of size $\lfloor{\frac{n^2}{32\,l\,s}}\rfloor$.
            On the other hand, we can always get homometric sets of size at least $\lfloor s/2 \rfloor$ and $\lfloor l/2 \rfloor$.

            Finally, optimizing over $\lfloor{\frac{n^2}{32\,l\,s}}\rfloor$, $\lfloor s/2 \rfloor$ and
            $\lfloor l/2 \rfloor$ gives the desired bound.
\end{proofHaircomb}

\section{Concluding remarks}

In this note we improved the lower bound on the maximum size of homometric sets in case of trees and haircomb trees
by considering  homometric sets having a very special structure, which in both cases implied that the two homometric sets in a considered pair
induce isomorphic forests, or more precisely a disjoint union of paths.
Moreover, in case of trees we showed that by restricting ourselves to homometric sets of this structure the lower bound we obtained is the best possible
in general.
On the other hand, our lower bound could be still improved by considering more general homometric sets.

Owing to its simple structure, trees, and especially haircomb trees seem to be an appropriate class of graphs to look at, if we want to improve
the upper bound on the size of homometric sets in case of general graphs. However, we have no good guess how a tree witnessing a sublinear upper bound (if it exists)
should look like, and we are prone to believe that the right bound in case of trees is linear in $n$.

\section{Acknowledgement}\label{sec:acknowledgement}
We thank to J\' anos Pach for introducing us the problem and
to Andres J. Ruiz-Vargas  for many helpful discussions. 

\bibliographystyle{plain}
\bibliography{ref}
\end{document}